\documentclass[a4paper,10pt]{amsart}

\usepackage{amsfonts,latexsym,rawfonts,amsmath,amssymb,amsthm, mathrsfs, lscape}
\usepackage{graphicx}
\usepackage[all]{xy}
\usepackage[english]{babel}
\usepackage[utf8]{inputenc}
\usepackage{xfrac}

\usepackage{cite}

\usepackage{array, tabularx}

\usepackage{bm}
\bmdefine{\bX}{X}
\bmdefine{\ba}{$\alpha$}

\usepackage{setspace}
\usepackage{textcomp}

\usepackage{color}

\usepackage{xparse}
\usepackage{soul}

\usepackage{graphicx}

\newcommand{\referenza}{}

\newtheorem{thm}{Theorem}[section]
\newtheorem*{thm*}{Theorem \referenza}
\newtheorem*{thm**}{Theorem}
\newtheorem{corollario}[thm]{Corollary}
\newtheorem*{cor*}{Corollary \referenza}

\newtheorem*{lem*}{Lemma \referenza}

\newtheorem*{clm*}{Claim \referenza}
\newtheorem{prop}[thm]{Proposition}
\newtheorem*{prop*}{Proposition \referenza}

\newtheorem*{ass*}{Assumption \referenza}

\newtheorem*{conj*}{Conjecture \referenza}

\newtheorem{rmk}[thm]{Remark}

\newtheorem{defi}[thm]{Definition}

\theoremstyle{plain}

\def \N {\mathbb N}

\def \R {\mathbb R}
\def \C {\mathbb C}
\def \Z {\mathbb Z}

\allowdisplaybreaks

\usepackage[hypertexnames=false,
backref=page,
    pdfpagemode=UseNone,
    breaklinks=true,
    extension=pdf,
    colorlinks=true,
    linkcolor=blue,
    citecolor=blue,
    urlcolor=blue,
]{hyperref}

\makeatletter
\@namedef{subjclassname@2020}{
  \textup{2020} Mathematics Subject Classification}
\makeatother

\begin{document}

\title{The prescribed Chern scalar curvature problem}

\author{Elia Fusi}
\address{	Dipartimento di Matematica "Giuseppe Peano"\\
	Universit\`a di Torino,
	Via Carlo Alberto 10,
	10123 Torino, Italy}
\email{elia.fusi@unito.it}

\keywords{Hermitian manifold, Chern scalar curvature, prescribed Chern scalar curvature problem, Chern-Yamabe problem.}
\subjclass[2020]{53C55 (primary), 53C21 (secondary).}

\begin{abstract}
The paper is an attempt to resolve the prescribed Chern scalar curvature problem. We look for solutions within the conformal class of a fixed Hermitian metric. We divide the problem in three cases, according to  the sign of the Gauduchon degree, that we analyse separately. In the case where the Gauduchon degree is negative, we prove that every non-identically zero and non-positive function is the Chern scalar curvature of a unique metric conformal to the fixed one.  Moreover, if there exists a balanced metric with zero Chern scalar curvature, we prove that every smooth function changing sign with negative mean value is the Chern scalar curvature of a metric conformal to the balanced one.
\end{abstract}

\maketitle

\section*{Introduction}
A classical problem in Riemannian Geometry is the prescribed scalar curvature problem. This problem consists in finding  a Riemannian metric on a differentiable manifold such that its scalar curvature coincides with a smooth function previously fixed. The prescribed scalar curvature problem was firstly studied by Kazdan and Warner in \cite{KazWar1} for 2-dimensional compact manifolds and in \cite{KazWar2} for compact manifolds with higher dimension.  The same authors concluded the study of the prescribed scalar curvature problem for compact manifolds in \cite{KazWar3}, also thanks to the proof of the Yamabe problem by Schoen in \cite{Sch}. 
Kazdan and Warner divided the class of compact manifolds into three subclasses depending on the set  of smooth functions that are scalar curvature of a Riemannian metric.\begin{thm**}[See \cite{KazWar3}, Theorem 5.1 and 6.4]The class of compact differentiable manifolds of dimension 2 can be divided into three subclasses:\begin{itemize}\item every smooth function that is positive somewhere is the scalar curvature of a Riemannian metric,\item the zero function and all the smooth functions changing sign are scalar curvature of a Riemannian metric,\item every smooth function that is negative somewhere is the scalar curvature of a Riemannian metric. \end{itemize}The class of compact differentiable manifolds of dimension at least 3 can be divided into three subclasses:\begin{itemize}\item every smooth function is the scalar curvature of a Riemannian metric,\item the zero function and all the smooth functions that are negative somewhere are scalar curvature of a Riemannian metric, \item every smooth function that is negative somewhere is the scalar curvature of a Riemannian metric.\end{itemize} \end{thm**}
In this paper, we study the prescribed Chern scalar curvature problem, namely the Hermitian analogue of the prescribed scalar curvature problem. It consists in proving the existence of a Hermitian metric on a complex manifold such that its scalar curvature with respect to the Chern connection coincides with a fixed smooth real-valued function. 
The motivation of our study comes from the recent studies on Chern-Yamabe problem (see \cite{AnCaSp}, \cite{CZ}, \cite{LM},\cite{Ho}, \cite{HS} and \cite{LU} in the almost-Hermitian setting). In particular, in \cite{Ho},  the prescribed Chern scalar curvature problem is also considered. The author uses the so called Chern-Yamabe flow, firstly introduced in \cite{CZ}, to prove, supposing there exists a balanced metric $\omega_0$ with negative Chern scalar curvature, that every negative function is the Chern scalar curvature of a metric conformal to $\omega_0$, see \cite[Theorem 1.1]{Ho}. The objective of our work is to generalize the result mentioned before and  study  the prescribed Chern scalar curvature more in depth. 

The paper is divided in two sections. In the first one, we define the tools that we used to prove our results, like the Chern Laplacian and the Gauduchon degree of a conformal class, and we recall the results concerning the Chern-Yamabe problem, see \cite{AnCaSp}. In the second section, we study the prescribed Chern scalar curvature on connected, compact complex manifolds with complex dimension at least 2. Our first approach is that of finding solutions within the conformal class of a fixed Hermitian metric $\omega$. In this setting, we see  that the resolution of the problem is equivalent to solving a non-linear elliptic PDE of 2nd order. Together with this equation, we find an obstruction that suggests a division of the problem into 3 cases depending on the sign of the Gauduchon degree of the conformal class of $\omega$. We analyse separately each case. 

In the case where $\Gamma(\{\omega\})<0$, we find a necessary condition on a smooth function in order for it to be Chern scalar curvature of a metric conformal to $\omega$. We observe that  the smooth functions that are non-identically zero and non-positive satisfy this condition. Our first result states that all these functions are Chern scalar curvature of a unique Hermitian metric in the conformal class of $\omega$.\begin{thm**}[See Theorem \ref{thm:2.5}]Let $M^n$ be a connected, compact complex manifold with $n\ge2$ endowed with a Hermitian metric $\omega$ such that $\Gamma(\{\omega\})<0$. If $g\in C^{\infty}(M,\R)\setminus\{0\}$ is such that $g\le0$, then $g$ is the Chern scalar curvature of a unique Hermitian metric conformal to $\omega$. \end{thm**} A stronger necessary condition is obtained in Proposition \ref{prop:2.6}. It allows us to construct a function which satisfies the first necessary condition but which cannot be Chern scalar curvature of any metric conformal to $\omega.$ Because of this, we try to find solutions within the set of metrics conformally equivalent to $\omega$. In this case, we find two sufficient conditions on a smooth function to be Chern scalar curvature of a metric conformally equivalent to $\omega$, see Proposition \ref{prop:2.8} and Proposition \ref{prop:2.9}.

In the case where $\Gamma(\{\omega\})=0$, first of all, we find two necessary conditions. Then, under a suitable assumption, we interpret the prescribed Chern scalar curvature equation as the Euler-Lagrange equation, with respect to the standard $L^2$ pairing, associated to an appropriate functional. This assumption concerns the existence of a balanced metric with zero Chern scalar curvature. Using  variational methods, we prove the following Theorem. \begin{thm**}[See Theorem \ref{thm:2.11}]Let $M^n$ be a connected, compact complex manifold with $n\ge2$. Suppose there exists a balanced metric $\omega$ with $S^{Ch}(\omega)=0$. If $g\in C^{\infty}(M,\R) $ changes sign on $M$ and $\int_Mg\frac{\omega^n}{n!}<0,$ then $g$ is the Chern scalar curvature of a Hermitian metric conformal to $\omega.$ \end{thm**} The case where $\Gamma(\{\omega\})>0$ is the most difficult among the three. In this case, we find a necessary condition and we prove a local result by applying the implicit function Theorem, see Proposition \ref{prop:2.14}.

After our work, there are many questions that remain open. One of the most important result that Kazdan and Warner used to conclude the prescribed scalar curvature problem is \cite[Theorem 2.1]{KazWar3}.  Also in our setting, it can be useful to know the orbit of a smooth function under the group of biholomorphisms in order to understand when Proposition \ref{prop:2.8}, Proposition \ref{prop:2.9} and Proposition \ref{prop:2.12} can be applied. Once we know this, we can understand better which functions are Chern scalar curvature of a metric conformally equivalent to $\omega$ and then conclude using \cite[Theorem 1.1]{Yan}. Another open question is the generalization of Theorem \ref{thm:2.11}, removing the hypothesis that there exists a balanced metric. The last open question is the problem of existence of metric with constant and positive Chern scalar curvature within a conformal class with positive Gauduchon degree.
\bigskip 

{\itshape Acknowledgements.} The author is very grateful to Daniele Angella for his supervision, support and stimulating discussions over months. Many thanks are also due to Francesco Pediconi for several discussions, suggestions and his interest on the subject. 

The author is supported by GNSAGA of INdAM.

\section{Preliminaries}
Let $M$ be a connected, compact complex manifold of dimension $n$. In our discussion, we will identify Hermitian metrics $h$ on $M$ with their fundamental (1,1)-forms $\omega=h(J-,-).$ For any Hermitian structure $(J,\omega)$ on $M$, we can consider the Chern connection $\nabla^{Ch}$, that is, the unique affine connection preserving both the Hermitian metric and the almost-complex structure, i.e. $\nabla^{Ch}h=0,\nabla^{Ch}J=0$, whose (0,1)-part coincides with the Cauchy-Riemann operator $\bar{\partial}$ associated with the holomorphic structure of $T^{1,0}M$. Given a Hermitian metric $\omega$ on $M$ and $\{z_1,\ldots,z_n\}$ local holomorphic coordinates, we can write $\omega=\sqrt{-1}\omega_{i\bar{j}}dz_i\wedge d\overline{z_j} $ where $\omega_{i\bar{j}}=h(\frac{\partial}{\partial z_i},\frac{\partial}{\partial \overline{z_j}})$. Then, we can define the Chern scalar curvature of $\omega$ as the function \[S^{Ch}(\omega)=tr_{\omega}Ric^{(1)}(\omega)=\omega^{i\bar{j}}\omega^{k\bar{l}}R_{i\bar{j}k\bar{l}},\] where $R_{i\bar{j}k\bar{l}}$ are the components of the Chern curvature tensor and $Ric^{(1)}(\omega)$ is called the first Chern-Ricci form and it is the $(1,1)$-form that, locally, can be written  as follows: \[Ric^{(1)}(\omega)=\sqrt{-1}\omega^{k\bar{l}}R_{i\bar{j}k\bar{l}}dz_i\wedge d\overline{z_j}.\] 

Given a Hermitian metric $\omega$ on $M$, we can define a differential operator called Chern Laplacian $\Delta_{\omega}^{Ch}$ associated to $\omega$ that acts on a smooth function $u\in C^{\infty}(M,\R)$ as follows:\[\Delta_{\omega}^{Ch}u=2\sqrt{-1}tr_{\omega}\bar{\partial}\partial u.\] Locally, we obtain that\[\Delta_{\omega}^{Ch}u=-2\omega^{i\bar{j}}\frac{\partial^2u}{\partial z_i\partial \overline{z_j}}, \ \ \forall u\in C^{\infty}(M,\R).\] It is well known that a Riemannian metric induces an inner product on the algebra of smooth differential forms on $M$. In the following, if $ \alpha$ and $\beta$ are $k$-forms on $M$ and $\omega $ is the fundamental form associated to a Hermitian metric $h$, we will set  $\omega(\alpha,\beta)=h(\alpha,\beta)$. 

Observe that the Chern Laplacian is an elliptic differential operator of 2nd order. On a compact Hermitian manifold $(M,\omega)$ both the Hodge Laplacian associated to the metric, i.e $\Delta_{\omega}=[d,d^*]$,  and the Chern Laplacian are defined. The difference between these two operators on smooth functions is quantified in a result due to Gauduchon. \begin{thm}[\cite{Gauduchon1}, p. 502-503]\label{thm:1.1}Let $(M^n,\omega)$ be a compact Hermitian manifold. Then, we have\[\Delta_{\omega}^{Ch}u=\Delta_{\omega}u+ \omega(du,\theta), \ \ \ \forall u\in C^{\infty}(M,\R),\] where $\theta$ is called the torsion 1-form associated to $\omega$ and it is defined by the relation: $d\omega^{n-1}=\theta\wedge \omega^{n-1}$. Moreover, we have that the formal adjoint of $\Delta_{\omega}^{Ch}$ acts on $u\in C^{\infty}(M,\R)$ as follows:\[(\Delta_{\omega}^{Ch})^*u=\Delta_{\omega}u-\omega(du,\theta)+ud^*\theta.\] \end{thm}\begin{rmk}Let $(M^n,\omega)$ a Hermitian manifold with $n\ge2$, we can consider the Lefschetz operator\[L\colon \bigwedge\nolimits^{\bullet}M\to \bigwedge\nolimits^{\bullet+2}M,  \ \ \ L-=-\wedge\omega.\] We know that \[L^{n-1}\colon\bigwedge\nolimits^1M\to\bigwedge\nolimits^{2n-1}M\] is an isomorphism. This fact implies that the torsion 1-form is well defined.\end{rmk} The Hermitian metric $\omega$ is called Gauduchon if $d^*\theta=0$. Instead, one says that $\omega$ is balanced if $\theta=0.$ If $\dim_{\C}M=n$, note that the condition $d^*\theta=0$ is equivalent to $\partial\bar{\partial}\omega^{n-1}=0$. That is because $\theta=J*d\omega^{n-1}$, where $*$ is the Hodge operator associated to $\omega$. Whereas, the condition $\theta=0$ is equivalent to $d\omega^{n-1}=0.$\begin{rmk}\label{rmk:1.3}Theorem \ref{thm:1.1} implies that the index of the Chern Laplacian and the index of the Hodge Laplacian coincide. This guarantees that the kernels of both the Chern Laplacian and his formal adjoint are 1-dimensional. Moreover, $\ker(\Delta_{\omega}^{Ch})=\R$. In general, $\ker((\Delta_{\omega}^{Ch})^*)$ does not coincide with the set of constant functions. That is, essentially, due to the presence of the additional term $ud^*\theta$ in the expression of $(\Delta_{\omega}^{Ch})^*u$, see \cite[p.388]{Gauduchon1}. \end{rmk} Remark \ref{rmk:1.3} suggests the definition of a particular function, called eccentricity function associated to a Hermitian metric.\begin{defi}[\cite{Gauduchon1}, Définition 2]\label{def:1.4}Let $(M,\omega)$ be a compact Hermitian manifold. A function $f_0\in C^{\infty}(M,\R)$ is the \emph{eccentricity function} associated to $\omega$ if $f_0\in \ker((\Delta_{\omega}^{Ch})^*)$ and $\langle f_0,1\rangle_{L^2(M)}=Vol(M,\omega).$\end{defi} Gauduchon proved that $f_0>0$. Moreover, he proved that $f_0=1$ if and only if $\omega$ is Gauduchon, see \cite[Théorème 2]{Gauduchon1}. As stated in Remark \ref{rmk:1.3}, the fact that $\dim(\ker(\Delta^{Ch}_{\omega})^*)=1$ easily guarantees that the eccentricity function associated to a Hermitian metric is unique.

A simple definition that we want to remember is the definition of conformal class and conformally equivalent metrics.\begin{defi}\label{def:1.5}Let $(M^n,\omega)$ be a Hermitian manifold. The conformal class of $\omega$ is \[\{\omega\}=\left\{\exp\left(\frac{2u}{n}\right)\omega \ \middle |\  u\in C^{\infty}(M,\R)\right\}.\] If $\omega$ and $\omega_1$ are Hermitian metrics on $M$, we say that $\omega_1$ is conformal to $\omega$ if $\omega_1\in\{\omega\}$. We will say that $\omega_1$ is conformally equivalent to $\omega$ if there exists $\varphi\in Aut(M)$, the group of biholomorphisms of $M$, such that $\varphi^*\omega_1\in\{\omega\}$. \end{defi} A result that will be very useful for  our study of the prescribed  Chern scalar curvature problem is the  characterization of the variation of the Chern scalar curvature after a conformal change. \begin{prop}[\cite{Gauduchon2}, p. 502]\label{prop:1.6}Let $(M^n,\omega)$ be a Hermitian manifold and $u\in C^{\infty}(M,\R)$. We have\[S^{Ch}\left(\exp\left(\frac{2u}{n}\right)\omega\right)=\exp\left(-\frac{2u}{n}\right)(\Delta_{\omega}^{Ch}u+S^{Ch}(\omega)).\]\end{prop}See \cite[Corollary 4.5]{LU} in the almost-Hermitian setting.

 The last quantity that we want to recall is the so called Gauduchon degree of a conformal class. To define it, we need to remember an important Theorem due to Gauduchon. The proof of this result is based on the properties, that we already mentioned, of the eccentricity function associated to a Hermitian metric.\begin{thm}[\cite{Gauduchon1}, Théorème 1]\label{thm:1.7}Let $(M^n,\omega)$ be a compact Hermitian manifold. There exists a Gauduchon metric with volume 1 within $\{\omega\}$. If $n\ge2$, this metric is also unique.\end{thm} Then, if $\dim_{\C}M\ge2$, we can choose a particular metric in each conformal class of a Hermitian metric. This allows us to define the Gauduchon degree of a conformal class.\begin{defi}[\cite{Gauduchon1}, I.17]\label{def:1.8}Let $(M^n,\omega)$ be a compact Hermitian manifold with $n\ge2$. Choose $\eta\in\{\omega\}$ the only Gauduchon metric with volume 1. The Gauduchon degree of $\{\omega\}$ is\[\Gamma(\{\omega\})=\int_MS^{Ch}(\eta)\frac{\eta^n}{n!}.\] \end{defi} \begin{rmk}\label{rmk:1.9}Note that\[\Gamma(\{\omega\})=\frac{1}{(n-1)!}\int_Mc_1^{BC}(M)\wedge\eta^{n-1},\] where $c_1^{BC}(M)=c_1^{BC}(K_M^*)$ is the first Bott-Chern class of $M$. Thanks to this relation, we have that the Gauduchon degree of a conformal class is an invariant depending only on the complex structure of $M$ and on the conformal class $\{\omega\}$. The definition we gave is nothing but a particular case of the so called degree of a line bundle with respect to a conformal class, see \cite[I.17]{Gauduchon1}. \end{rmk} Before we start analysing the prescribed Chern scalar curvature problem, it will be useful to study the "constant" case, that is the Chern-Yamabe problem. This problem consists in proving the existence of a metric conformal to a fixed one with constant Chern scalar curvature. The Chern-Yamabe problem was firstly studied and, partially, resolved by Angella, Calamai and Spotti. The solution of the Chern-Yamabe problem can be reduced to solving  a non-linear elliptic PDE of 2nd order. Studying this equation, the authors were able to prove an important result.\begin{thm}[\cite{AnCaSp}, Theorem 3.1 and Theorem 4.1 ]\label{thm:1.10}
Let $M^n$ be a connected, compact complex manifold with $n\ge2$ endowed with a Hermitian metric $\omega$ with $\Gamma(\{\omega\})\le0$. Then, there exists a unique Hermitian metric conformal to $\omega$ with Chern scalar curvature that coincides with $\Gamma(\{\omega\})\le0.$ \end{thm}

 So, if $(M^n,\omega)$ is a connected, compact  Hermitian manifold with $n\ge2$, at least in the case where $\Gamma(\{\omega\})\le0$, we can choose, within the conformal class, a metric with constant Chern scalar curvature. Unfortunately, in the case where $\Gamma(\{\omega\})>0,$ we do not know  if we can choose a metric conformal to the fixed one with constant Chern scalar curvature. Surely, the uniqueness does not  hold anymore. In \cite{AnCaSp}, the authors construct an example in which they prove that there exist at least two, non homothetic, metrics with constant Chern scalar curvature within the conformal class of a particular Hermitian metric.

\section{The prescribed Chern scalar curvature problem}
The prescribed Chern scalar curvature problem consists in proving the existence of a Hermitian metric on a connected, compact complex manifold of dimension at least $2$ such that its Chern scalar curvature coincides with a fixed smooth function. 

Our approach to the resolution of the problem is that of finding solutions within the conformal class of a fixed Hermitian metric. Therefore,  we fix a Hermitian metric $\omega$ and $g\in C^{\infty}(M,\R)$. Thanks to the Proposition \ref{prop:1.6}, a metric conformal to $\omega$ has Chern scalar curvature coinciding with $g$ if and only if  the equation \begin{equation}\label{eq:1}\Delta_{\omega}^{Ch}u+S^{Ch}(\omega)=g\exp\left(\frac{2u}{n}\right)
\end{equation} admits a solution. On the other hand, we can choose the unique Gauduchon metric $\eta\in\{\omega\}$ with volume 1 as reference metric and rewrite the equation (\ref{eq:1}) as a  function of $\eta.$ Integrating that equation over $M$, we obtain the condition\begin{equation}\label{eq:2}\Gamma(\{\omega\})=\int_Mg\exp\left(\frac{2u}{n}\right)\frac{\eta^n}{n!}.\end{equation} The condition we just found suggests a division of the problem into three cases, depending on the sign of  $\Gamma(\{\omega\}). $ We will analyse each of these cases separately. Note that, in this problem, the Gauduchon degree of the conformal class plays the same role as the first eigenvalue of the linear part of the operator defining the prescribed scalar curvature equation studied in \cite{KazWar2}.
\subsection{Case $\Gamma(\{\omega\})<0$.}

In this case, the condition (\ref{eq:2}) implies that the function $g$ must be negative somewhere on $M$. Furthermore, thanks to Theorem \ref{thm:1.10}, we can choose the unique Hermitian metric in $\{\omega\}$ with constant Chern scalar curvature equal to $\Gamma(\{\omega\})<0.$ In the following, we will indicate this metric with $\omega$ and with $f_0$ its eccentricity function. We can rewrite the equation (\ref{eq:1}) as follows:\begin{equation}\label{eq:3}\Delta_{\omega}^{Ch}u+\Gamma(\{\omega\})=g\exp\left(\frac{2u}{n}\right).\end{equation}Multiplying  the equation (\ref{eq:3}) by $\exp(-\frac{2u}{n})$, we can use the formula \begin{equation}\label{eq:4}\exp\left(-\frac{2u}{n}\right)\Delta^{Ch}_{\omega}u=-\frac{n}{2}\Delta^{Ch}_{\omega}\left(\exp\left(-\frac{2u}{n}\right)\right)-\frac{4}{n}\exp\left(-\frac{2u}{n}\right)\omega(du,du).\end{equation} to obtain a new equation. By multiplying this new one by $f_0 $ and integrating on $M$, we obtain a new necessary condition on $g$, that is\begin{equation}\label{eq:*}\int_Mgf_0\frac{\omega^n}{n!}<0.\tag{*}\end{equation}  We already noticed that $f_0>0$ on $M$, so all non-identically zero and non-positive smooth functions
satisfy this condition. Then, our first objective is to prove that such functions $g$ are Chern scalar curvature of a metric conformal to $\omega.$ Using the continuity method, we prove the following  generalization of \cite[Theorem 1.1]{Ho}, removing the hypothesis that $\omega$ is balanced.\begin{thm}\label{thm:2.1}Let $M^n$ be a connected, compact complex manifold with $n\ge2$ endowed with a Hermitian metric $\omega$ such that $\Gamma(\{\omega\})<0.$ If $g\in C^{\infty}(M,\R),$ $g<0$, then $g$ is the Chern scalar curvature of a unique metric conformal to $\omega$.\end{thm} \begin{proof}The proof is essentially the same as in \cite[Theorem 4.1]{AnCaSp}. First of all, suppose that $\omega$ is such that $S^{Ch}(\omega)=\Gamma(\{\omega\})<0$. The uniqueness is a direct consequence of the maximum principle. We use the continuity method to obtain the existence. So, define, for any $ t\in [0,1]$, the equation\begin{equation}\label{eq:$3_t$}\Delta_{\omega}^{Ch}u+t\Gamma(\{\omega\})-g\exp\left(\frac{2u}{n}\right)+(1-t)g=0,\tag{$3_t$}\end{equation} and \[T=\{t\in[0,1] \ | \  \text{(\ref{eq:$3_t$}) has a solution in } \ C^{2,\alpha}(M)\}.\] Observe that $T\ne\emptyset$ because $u=0$ is a solution of $(3_0)$. 

Fix $t\in[0,1]$ and define $G\colon C^{2,\alpha}(M)\to C^{0,\alpha}(M) $ such that, for any $ \varphi\in C^{2,\alpha}(M)$, \[G(\varphi)=\Delta_{\omega}^{Ch}\varphi+tS^{Ch}(\omega)-g\exp\left(\frac{2\varphi}{n}\right)+(1-t)g.\] Suppose that $u\in C^{2,\alpha}(M)$ is a solution of (\ref{eq:$3_t$}). We have that \[d_uG=\Delta_{\omega}^{Ch}-\frac{2}{n}g\exp\left(\frac{2u}{n}\right)Id.\] Obviously, $d_uG$ is an elliptic differential operator. By the maximum principle, it follows that $d_uG$ is injective. On the other hand, the index of $d_uG$ coincides with the index of the Hodge Laplacian. Therefore, $(d_uG)^*$ must be injective and then $d_uG$ must be surjective too, so $d_uG$ is invertible. Applying the implicit function Theorem, we obtain that $T$ is open. 

If we can prove that $T$ is closed too, automatically, we obtain that $T=[0,1]$. In particular, this fact guarantees  that $(3_1)$ has a solution in $C^{2,\alpha}(M).$ But $(3_1)$ coincides with (\ref{eq:3}) and then we obtain the existence we are looking for.  We need to find some a priori estimates on solutions of (\ref{eq:$3_t$}) in order to  prove that $T$ is closed. Then, let $u\in C^{2,\alpha}(M)$ be a solution of (\ref{eq:$3_t$}) and let $p,q\in M$ be, respectively, the maximum and minimum point of $u$. We have  \[-g(p)\exp\left(\frac{2u(p)}{n}\right)=-(\Delta_{\omega}^{Ch}u)(p)-tS^{Ch}(\omega)(p)-(1-t)g(p)\]\[\le-tS^{Ch}(\omega)(p)-(1-t)g(p)\le- \min_MS^{Ch}(\omega)-\min_Mg.\]Then \[\exp\left(\frac{2u(p)}{n}\right)\le-\frac{1}{g(p)}\left(- \min_MS^{Ch}(\omega)-\min_Mg\right)=C(M,\omega,g).\]Similarly,  \[-g(q)\exp\left(\frac{2u(q)}{n}\right)\ge\min_{t\in[0,1]}\left(t\min_M\left(-S^{Ch}(\omega)\right)+(1-t)\min_M\left(-g\right)\right)=C'(M,\omega,g).\] These two inequalities imply that \[\lVert u\rVert_{L^{\infty}(M)}\le K,\] where $K=K(M,\omega,g)>0$. Thanks to this inequality, iterating the Calder\'on-Zygmund inequality and using the Sobolev embedding, we obtain  that  there exists a constant $K'=K'(M,\omega,g)>0$ such that \[\lVert u\rVert_{C^{3,\alpha}(M)}\le K'.\]Then, choose $\{t_k\}_{k\in\N}\subset T$ such that $t_k\to t_{\infty}$ as $k\to+\infty$. We take $\{u_k\}_{k\in\N}\subset C^{2,\alpha}(M)$ such that $u_k$ is a solution of $(3_{t_k})$, $\forall k\in \N$.  Thanks to the estimates above and to Ascoli-Arzelà Theorem, we have  that there exists $u\in C^{2,\alpha}(M)$ such that $u_k\to u$ in $C^{2,\alpha}(M)$. We see that $u$ is a solution of $(3_{t_{\infty}})$. Then $T$ is closed. Finally, if $u\in C^{2,\alpha}(M)$ is a solution of $(3_1)$, we can prove that $u\in C^{\infty}(M,\R)$ by using Schauder estimates. \end{proof} The a priori estimates in the proof of Theorem \ref{thm:2.1} do not hold if $g$ is zero somewhere in $M$. So, we have to find some alternative method to solve the equation (\ref{eq:3}) in this case.
 \begin{defi} Let $(M^n,\omega)$ be a connected, compact  Hermitian manifold with $n\ge2$. We say that $u_-,u_+\in W^{2,p}(M) $, $p>n$, are, respectively, a subsolution and a supersolution for the equation (\ref{eq:1}) if \[\Delta_{\omega}^{Ch}u_{-}\le g\exp\left(\frac{2u_{-}}{n}\right)-S^{Ch}(\omega) ,\ \ \ \Delta_{\omega}^{Ch}u_{+}\ge g\exp\left(\frac{2u_{+}}{n}\right)-S^{Ch}(\omega).\]\end{defi} We prove the following Theorem which states that the existence of solution of (\ref{eq:1}) is equivalent to the existence of both a subsolution and a supersolution.
\begin{thm}\label{thm:2.3}Let $(M^n,\omega)$ be a connected, compact  Hermitian manifold with $n\ge2$. Suppose that we have $g\in C^{\infty}(M,\R),$ $p>n$ and $u_-,u_+\in W^{2,p}(M)$, respectively, a subsolution and a supersolution of (\ref{eq:1}) such that $u_-\le u_+.$ Then, there exists $u\in C^{\infty}(M,\R)$ a solution of (\ref{eq:1}) such that $u_-\le u\le u_+.$\end{thm}
\begin{proof}
The proof of this result is similar to the proof of  \cite[Lemma 2.6]{KazWar2}. We recall here the main ideas. We set $f(x,u)=g\exp(\frac{2u}{n})-S^{Ch}(\omega)$.
 We define $u_0=u_+$ and, $\forall k\in \N$, $k\ge1$, the function $u_k$  as the only solution of the equation\[\Delta_{\omega}^{Ch}u_k+Ku_k=Lu_{k}=f(x,u_{k-1})+Ku_{k-1}, \] where $K>0$ is an appropriate constant. The existence of the functions $u_k$ is guaranteed by the invertibility of the operator $L\colon W^{2,p}(M)\to L^p(M)$, which is a consequence of the standard theory of elliptic PDEs.
 Using the maximum principle, we prove that, $\forall k\in \N$, $u_-\le u_{k-1}\le u_k\le u_+$. On the other hand, using  the inequality 
\[\lVert v\rVert_{W^{2,p}(M)}\le C\lVert Lv\rVert_{L^p(M)},\,\,\, \forall v\in W^{2,p}(M),\] 
due to the invertibility of $L$, the Sobolev embeddings and the Ascoli-Arzelà theorem, we prove that $u_k\to u$ in $C^0(M)$.
 From this fact, we prove that $u_k\to u$ in $W^{2,p}(M)$, so \[Lu=f(x,u)+Ku.\]
This fact implies that $u$ is a solution of (\ref{eq:1}).
 Using the Calder\'on-Zygmund inequality, we obtain that $u\in C^{\infty}(M,\R).$\end{proof}
So, thanks to Theorem \ref{thm:2.3}, in order to find a solution of (\ref{eq:1}), we can concentrate on finding both a subsolution $u_-$ and a supersolution $u_+$ such that $u_-\le u_+.$ Nevertheless, in the case we are analysing in this subsection, it can be proved  that this last condition is equivalent to the existence of only a supersolution.
\begin{corollario}\label{cor:2.4}Let $(M^n,\omega)$ be a connected, compact  Hermitian manifold with $S^{Ch}(\omega)=\Gamma(\{\omega\})<0$ and $n\ge2$. Suppose that we have $g\in C^{\infty}(M,\R),$ $p>n$ and $u_+\in W^{2,p}(M)$ a supersolution of (\ref{eq:3}). Then, there exists  a subsolution $u_-\in W^{2,p}(M)$ such that $u_-\le u_+$. \end{corollario}
\begin{proof} Thanks to the necessary condition (\ref{eq:*}), we can suppose that $g\in C^{\infty}(M,\R)$ is negative somewhere on $M$. We can choose $u_-\in\R$ satisfying the inequality\[u_-\le\frac{n}{2}\log\left(\frac{\Gamma(\{\omega\})}{\min_Mg}\right)\] and obtain a subsolution of (\ref{eq:3}). But, if $u_+$ is a supersolution of (\ref{eq:3}), we know that $u_+\in W^{2,p}(M)$.  Using the Sobolev embeddings, $u_+$ is continuous and then  bounded on $M$. So, we can choose appropriately the constant $u_-$ so that the inequality $u_-\le u_+$ is satisfied. \end{proof}
 Corollary \ref{cor:2.4}  allows us to prove the following  generalization of Theorem \ref{thm:2.1}.
 \begin{thm}\label{thm:2.5}Let $M^n$ be a connected, compact complex manifold with $n\ge2$ endowed with a Hermitian metric $\omega$ such that $\Gamma(\{\omega\})<0.$ If $g\in C^{\infty}(M,\R)\setminus\{0\}$ is such that $g\le0$, then $g$ is the Chern scalar curvature of a unique metric conformal to $\omega$.\end{thm} \begin{proof}Consider  $\phi\in C^{\infty}(M,\mathbb{R})$ such that\[\Delta_{\omega}^{Ch}\phi=g-\frac{1}{Vol(M,\omega)}\int_M gf_0\frac{\omega^n}{n!}.\]Let $k_1,k_2\in \R$ be two constants that we will choose later and define $u_+=k_1\phi+k_2$. We have that \[\begin{split}\Delta_{\omega}^{Ch}u_++\Gamma(\{\omega\})-g\exp\left(\frac{2u_+}{n}\right)=k_1g &-\frac{k_1}{Vol(M,\omega)}\int_Mgf_0\frac{\omega^n}{n!}\\ &+\Gamma(\{\omega\}) -g\exp\left(\frac{2(k_1\phi+k_2)}{n}\right).	\end{split}\] It is sufficient to choose $k_1,k_2 $ such that \[\begin{cases}k_1\ge\frac{Vol(M,\omega)\Gamma(\{\omega\})}{\int_Mgf_0\frac{\omega^n}{n!}}>0 \\ k_2\ge \frac{n}{2}\log(k_1)-k_1\phi\end{cases}\] to obtain a supersolution of (\ref{eq:3}). The uniqueness of the metric we found is due to the maximum principle.\end{proof} So, it remains only to  understand when a smooth function  which changes sign on $M$ and satisfies the necessary condition (\ref{eq:*}) can be the Chern scalar curvature of a metric conformal to $\omega$. \begin{prop}\label{prop:2.6}Let $M^n$ be a connected, compact complex manifold with $n\ge2$ endowed with a Hermitian metric $\omega$ such that $S^{Ch}(\omega)=\Gamma(\{\omega\})<0.$ Suppose that $g\in C^{\infty}(M,\R) $ such that $\int_Mgf_0\frac{\omega^n}{n!}<0$.\begin{enumerate}
\renewcommand{\labelenumi}{\alph{enumi})}
\item If $g$ is the Chern scalar curvature of a metric conformal to $\omega$, then the unique solution of the equation \begin{equation}\label{eq:5}\Delta_{\omega}^{Ch}\psi-\frac{2}{n}\Gamma(\{\omega\})\psi=-\frac{2}{n}g\tag{$5$}\end{equation} must be positive.\item Let  $g_1\in C^{\infty}(M,\R)$  be such that $g_1\le g$ and  $\lambda>0.$ If $g$ is the Chern scalar curvature of a metric conformal to $\omega$, then both $g_1$ and $\lambda g$ will be.\item
   There exists a constant $c(g)\in[-\infty,0)$ such that the equation \begin{equation}\label{eq:$6_c$}\Delta_{\omega}^{Ch}u+c-g\exp\left(\frac{2u}{n}\right)=0\tag{$6_c$}\end{equation} admits a solution, $\forall c\in(c(g),0).$\end{enumerate}\end{prop}\begin{proof}\begin{enumerate} 
\renewcommand{\labelenumi}{\alph{enumi})}
\item Suppose that $u\in C^{\infty}(M,\R)$ is a solution of (\ref{eq:3}). Set $v=\exp(-\frac{2u}{n})$. Easily, we see that $v$ satisfies the equation\[\Delta_{\omega}^{Ch}v-\frac{2}{n}\left(-g+\Gamma(\{\omega\})v-n\frac{\omega(dv,dv)}{v}\right)=0.\]  If $\psi \in C^{\infty}(M,\R)$ is the unique solution of (\ref{eq:5}), we see that \[\Delta_{\omega}^{Ch}(\psi-v)-\frac{2}{n}\Gamma(\{\omega\})(\psi-v)=2\frac{\omega(dv,dv)}{v} >0.\] By the maximum principle, we have $\psi\ge v>0$.\item Note that, if $u_+ $ is a supersolution of (\ref{eq:3})  for $g$, it will be a supersolution of (\ref{eq:3}) for $g_1$ too. Moreover, we have $S^{Ch}(\lambda\omega)=\lambda^{-1}S^{Ch}(\omega)$, $\forall \lambda>0$.  From this, the assertion follows.\item 
Using the same arguments used to prove Theorem \ref{thm:2.3}, it follows that the equation (\ref{eq:$6_c$}) admits a solution if and only if it admits a supersolution. As observed above, if $c>\tilde{c}$, a supersolution of (\ref{eq:$6_c$}) is a supersolution of $(6_{\tilde{c}})$ too. The existence of a supersolution of (\ref{eq:$6_c$}) is equivalent to the existence of a positive solution of the following inequality:\begin{equation}\label{eq:7}\Delta_{\omega}^{Ch}v-\frac{2}{n}\left(-g+cv-n\frac{\omega(dv,dv)}{v}\right)\le 0.\tag{7}\end{equation} Choose $\phi\in C^{\infty}(M,\R)$ such that\[\Delta_{\omega}^{Ch}\phi=\frac{2}{n}\left(-g+\frac{1}{Vol(M,\omega)}\int_Mgf_0\frac{\omega^n}{n!}\right).\] Considering $a\in\R$ such that\[ a\ge\max_M\left\{-\frac{2nVol(M,\omega)}{\int_Mgf_0\frac{\omega^n}{n!}}\omega(d\phi,d\phi)-\phi\right\}\] and $c\ge \frac{\int_Mgf_0\frac{\omega^n}{n!}}{2Vol(M,\omega)\min_M\{\phi+a\}}$, we verify that $\phi+a$ satisfies (\ref{eq:7}). This fact implies that \begin{equation}\label{eq:8}c(g)\le\frac{\int_Mgf_0\frac{\omega^n}{n!}}{2Vol(M,\omega)\min_M\{\phi+a
\}}<0.\tag{8}\end{equation}\end{enumerate}\end{proof}
\begin{rmk}In particular, the first part of the above Proposition  guarantees a stronger necessary condition than (\ref{eq:*}). Indeed, if we multiply (\ref{eq:5}) by $f_0$ and then integrate on $M$, we obtain that\[\int_Mgf_0\frac{\omega^n}{n!}<0.\]Using the same arguments used to prove Theorem \ref{thm:2.5}, observe that $c(g)=-\infty$ for all $g\in C^{\infty}(M,\R)\setminus\{0\}$ such that $g\le0$. However, we do not know if these functions are the only ones that satisfy this property. \end{rmk}
Thanks to part a) of Proposition \ref{prop:2.6}, we can prove the existence of smooth functions that satisfy the condition (\ref{eq:*}) which can not be Chern scalar curvature of a metric conformal to $\omega.$\begin{corollario}\label{cor:2.8}Let $M^n$ be a connected, compact complex manifold with $n\ge2$ endowed with a Hermitian metric $\omega$ such that $S^{Ch}(\omega)=\Gamma(\{\omega\})<0.$ Then, there exists $g\in C^{\infty}(M,\R)$ with $\int_Mgf_0\frac{\omega^n}{n!}<0$ such that it cannot be the Chern scalar curvature of any metric conformal to $\omega$.\end{corollario}\begin{proof}Fix $\psi\in C^{\infty}(M,\R)\setminus\{0\}$ such that $\int_M\psi f_0\frac{\omega^n}{n!}=0.$ Choosing $0<a<-\min_M\psi$, we have that $\psi+a$ changes sign and we can define \[g=\frac{n}{2}\left(-\Delta_{\omega}^{Ch}\psi+\frac{2}{n}\Gamma(\{\omega\})(\psi+a)\right)\in C^{\infty}(M,\mathbb{R}).\]  We can easily see that\[\int_Mgf_0\frac{\omega^n}{n!}=\Gamma(\{\omega\})aVol(M,\omega)<0\] and \[\Delta_{\omega}^{Ch}(\psi+a)-\frac{2}{n}\Gamma(\{\omega\})(\psi+a)=-\frac{2}{n}g.\] The assertion now follows from this and part a) of Proposition \ref{prop:2.6}.\end{proof}\begin{rmk}\label{rmk: 2.9}The proof of Corollary \ref{cor:2.8} suggests a method whereby we can construct explicit examples of smooth functions that satisfies condition (\ref{eq:*}) which cannot be the Chern scalar curvature of a metric conformal to $\omega$. Indeed, we can choose $\psi'\in C^{\infty}(M,\mathbb{R})$ which is not constant and define $\psi=\psi'+ k$, where $k\in \mathbb{R}$ such that $\int_M\psi f_0\frac{\omega^n}{n!}=0$. Then, following the steps of the proof, we can obtain a function which satisfies (\ref{eq:*}) that cannot be the Chern scalar curvature of a metric conformal to $\omega.$  \end{rmk}
Corollary \ref{cor:2.8} states that not all the functions satisfying the condition (\ref{eq:*}) can be Chern scalar curvature of  a metric conformal to $\omega.$ So, we concentrate on searching solutions of the Chern scalar curvature problem within the set of  metrics conformally equivalent to $\omega$, namely, we search a biholomorphism $\varphi$ of $M$ and a function $u \in C^{\infty}(M,\mathbb{R})$ such that $S^{Ch}(\exp(\frac{2u}{n})\varphi^*\omega)=g$. Using the same methods we used above, this problem is equivalent to find $\varphi\in Aut(M)$ such that the equation\begin{equation}\Delta_{\omega}^{Ch}u+\Gamma(\{\omega\})=(g\circ\varphi)\exp\left(\frac{2u}{n}\right)\tag{9}\end{equation} admits a solution. We can observe immediately that the automorphism $\varphi$ must be such that \[\int_M(g\circ\varphi)f_0\frac{\omega^n}{n!}<0.\] This condition is a first obstruction to our study. We do not know if, in general, there exists an automorphism of $M$ such that the condition above can be satisfied. Nevertheless, we concentrate on searching sufficient conditions on $g$ so that  it is the Chern scalar curvature of a metric conformally equivalent to $\omega.$ The part c) of the Proposition \ref{prop:2.6} implies directly a sufficient condition.\begin{prop}\label{prop:2.8}Let $M^n$ be a connected, compact complex manifold with $n\ge2$ endowed with a Hermitian metric $\omega$ such that $S^{Ch}(\omega)=\Gamma(\{\omega\})<0.$ Suppose that $g\in C^{\infty}(M,\R)$ is negative somewhere on $M$. Then, $g$ is the Chern scalar curvature of a metric conformally equivalent to $\omega$ if there exists $\varphi\in Aut(M)$ such that $\int_M(g\circ\varphi) f_0\frac{\omega^n}{n!}<0$ and $c(g\circ\varphi)<\Gamma(\{\omega\})$. \end{prop} Unfortunately, we do not have an explicit expression for the constant $c(g\circ\varphi).$ On the other hand, we have some estimates on that constant. These estimates allow us to improve the Proposition \ref{prop:2.8}. \begin{prop}\label{prop:2.9}Let $M^n$ be a connected, compact complex manifold with $n\ge2$ endowed with a Hermitian metric $\omega$ such that $S^{Ch}(\omega)=\Gamma(\{\omega\})<0.$ Suppose that $g\in C^{\infty}(M,\R)$ is negative somewhere on $M$. Then, $g$ is the Chern scalar curvature of a metric conformally equivalent to $\omega$ if there exist $\varphi\in Aut(M)$ and $p>n$ such that $\int_M(g\circ\varphi) f_0\frac{\omega^n}{n!}<0$ and \[\left\lVert g\circ\varphi-\frac{1}{Vol(M,\omega)}\int_M(g\circ\varphi)f_0\frac{\omega^n}{n!}\right\rVert_{L^p(M)}< C,\] where $C=C(M,\omega,g,p)>0$.  \end{prop}\begin{proof}Suppose that there exists $\varphi\in Aut(M)$ such that $\int_M(g\circ\varphi)f_0\frac{\omega^n}{n!}<0.$ Thanks to  Proposition \ref{prop:2.8}, it is sufficient to impose that $c(g\circ\varphi)<\Gamma(\{\omega\}).$ Using (\ref{eq:8}), we obtain that it is sufficient to impose the following condition\[\frac{\int_M(g\circ\varphi)f_0\frac{\omega^n}{n!}}{2Vol(M,\omega)\min_M\{\phi+a\}}<\Gamma(\{\omega\}),\] where $\Delta_{\omega}^{Ch}\phi=\frac{2}{n}(-g\circ\varphi+\frac{1}{Vol(M,\omega)}\int_M(g\circ\varphi)f_0\frac{\omega^n}{n!})$ and \[a=\max_M\left\{-\frac{2nVol(M,\omega)}{\int_M(g\circ\varphi)f_0\frac{\omega^n}{n!}}\omega(d\phi,d\phi)-\phi\right\}.\]With these choices, we have\[\min_M\{\phi+a\}\le-\frac{2nVol(M,\omega)}{\int_M(g\circ\varphi)f_0\frac{\omega^n}{n!}}\max_M\omega(d\phi,d\phi).\] Then, using this inequality, we obtain \[\frac{\int_M(g\circ\varphi)f_0\frac{\omega^n}{n!}}{2Vol(M,\omega)\min_M\{\phi+a\}}\le -\frac{(\int_M(g\circ\varphi)f_0\frac{\omega^n}{n!})^2}{4nVol(M,\omega)^2\max_M\omega(d\phi,d\phi)}.\] By the standard elliptic theory, we know that there exists, if $p>n$, a constant $K=K(M,\omega,p)$ such that \[\lVert du\rVert_{C^0(M)}\le K\lVert\Delta_{\omega}^{Ch}u\rVert_{L^p(M)}.\] So, by straightforward calculations, it is sufficient to impose that there exists $p>n$ such that \[\left\lVert g\circ\varphi-\frac{1}{Vol(M,\omega)}\int_M(g\circ\varphi)f_0\frac{\omega^n}{n!}\right\rVert_{L^p(M)}<-\frac{\sqrt{n}\int_M(g\circ\varphi)f_0\frac{\omega^n}{n!}}{4KVol(M,\omega)\sqrt{-\Gamma(\{\omega\})}}=C.\] \end{proof} Unfortunately, the condition found in the Proposition above is not easily checkable. As we can see from the proof, the constant $C$ depends on easily computable quantities like $Vol(M,\omega)$ and $\Gamma(\{\omega\})$ but also on $K$, that is the norm of the embedding $W^{2,p}(M)\hookrightarrow C^{1}(M)$, which is not explicit, in general, for compact manifolds. 

\subsection{Case $\Gamma(\{\omega\})=0$} 

 In this case, the condition (\ref{eq:2}) guarantees that $g\in C^{\infty}(M,\R)$ must either be identically zero or change sign on $M$. Thanks to Theorem \ref{thm:1.10}, we can choose a metric conformal to $\omega$ such that its Chern scalar curvature is identically zero. We will indicate such metric with $\omega$ and with $f_0$ its eccentricity function. The equation (\ref{eq:3})  can be rewritten as follows:\begin{equation}\label{eq:10}\Delta_{\omega}^{Ch}u=g\exp\left(\frac{2u}{n}\right).\tag{10}\end{equation} As we already did in the previous case, we can multiply by $\exp(-\frac{2u}{n})$ the equation (\ref{eq:10}), use the formula (\ref{eq:4}), multiply by $f_0$ and then integrate on $M$ in order to obtain the following additional necessary condition:\[\int_Mgf_0\frac{\omega^n}{n!}<0.\] In general, we can guess if the equation (\ref{eq:3}) can be interpreted as the Euler-Lagrange equation, with respect to a pairing, associated to an appropriate functional.  Unfortunately, considering the $L^2$ standard pairing, this interpretation is not possible unless the metric we consider is balanced.\begin{prop}[\cite{AnCaSp}, Proposition 5.3]Let $M^n$ be a connected, compact complex manifold endowed with a Hermitian metric $\omega$. The 1-form on $C^{\infty}(M,\R)$\[\alpha\colon h\mapsto\int_Mh\omega(df,\theta)\frac{\omega^n}{n!}\] is never closed, beside the case when it is identically zero, which happens if and only if $\omega$ is balanced. 
It follows that equation (\ref{eq:3}) can be seen as the Euler-Lagrange equation for the standard $L^2$ pairing if and only if $\omega$ is balanced. In this case, the functional takes the form:\[\mathcal{F}(u)=\int_M(\omega(du,du)+ S^{Ch}(\omega)u)\frac{\omega^n}{n!}, \ \ \ \forall u\in W^{1,2}(M)\] subject to the constraint \[ \int_Mg\exp\left(\frac{2u}{n}\right)\frac{\omega^n}{n!}=0.\]
\end{prop}  As observed in \cite[Remark 5.4]{AnCaSp}, we do not know if the functional $\mathcal{F}$ is, in general, bounded below. Clearly, this property is satisfied by imposing that $S^{Ch}(\omega)=0.$ Then, in the hypothesis that $\omega$ is balanced and $S^{Ch}(\omega)=0$, we can apply the variational methods to find a solution of equation (\ref{eq:3}). \begin{thm}\label{thm:2.11}Let $M^n$ be a connected, compact complex manifold with $n\ge 2$. Suppose that there exists a balanced metric $\omega$ such that $S^{Ch}(\omega)=0$. Then, all the functions $g\in C^{\infty}(M,\R)$ changing sign and such that $\int_Mg\frac{\omega^n}{n!}<0$ are Chern scalar curvature of a metric conformal to $\omega$.\end{thm}\begin{proof}Define \[B=\left\{\phi\in W^{1,2}(M)\middle| \ \ \int_Mg\exp\left(\frac{2\phi}{n}\right)\frac{\omega^n}{n!}=0, \int_M\phi\frac{\omega^n}{n!}=0 \right\}\] and note that $B\ne\emptyset$ because $g$ changes sign on $M$. By straightforward calculations, it follows that $\exp(-)\colon W^{1,2}(M)\to W^{1,2}(M)$ is continuous. Define $a=\inf_B\mathcal{F}$. We consider a  minimizing sequence $\{v_k\}_{k\in\N}\subset W^{1,2}(M)$ such that $\{\mathcal{F}(v_k)\}_{k\in\N}$ goes to $a$ decreasingly. We can choose $v_0\in B$ and define $b=\mathcal{F}(v_0)\ge0$. Without loss of generality, we can suppose that $\mathcal{F}(v_k)\le b$, $\forall k\in \N$. Thanks to Poincaré inequality, we verify that\[ \lVert v_k\rVert^2_{W^{1,2}(M)}\le C \lVert dv_k\rVert^2_{L^2(M)}=C\mathcal{F}(v_k)\le Cb.\] So, $\{v_k\}_{k\in\N}$ is bounded in $W^{1,2}(M)$ and we know that a bounded set is weakly compact, thanks to Banach-Alaouglu-Bourbaki Theorem. So, up to subsequences, there exists a function $v\in W^{1,2}(M)$ such that    $v_k\rightharpoonup v$ in $W^{1,2}(M).$ This fact implies straightforwardly, using the continuity of $\exp(\frac{2-}{n})$,  that $v\in B$. Moreover, using again the Poincarè inequality, we know that $\lVert-\rVert_{W^{1,2}(M)}$ is equivalent to $\sqrt{\mathcal{F}(-)}.$ So, using the lower semicontinuity of the weak convergence with respect to $\sqrt{\mathcal{F}(-)}$, we have \[\sqrt{\mathcal{F}(v)}\le\liminf_{k\to+\infty}\sqrt{\mathcal{F}(v_k)}\le \sqrt{\mathcal{F}(v_k)}, \ \ \forall k\in\mathbb{N}.\]So, $v\in B$ minimizes the functional $\mathcal{F}$ on $B$. By standard Lagrange multiplier theory, we have two constant $\lambda,\mu\in \R$ such that \begin{equation}\label{eq:11}\int_M\left(\omega(d\varphi,dv)+\frac{2}{n}\lambda\varphi g\exp\left(\frac{2v}{n}\right)+\frac{2}{n}\mu\varphi\right)\frac{\omega^n}{n!}=0, \ \ \forall \varphi\in W^{1,2}(M).\tag{11}\end{equation} Choosing $\varphi=1$, we obtain automatically that $\mu=0.$ Moreover, if we choose $\varphi=\exp(-\frac{2v}{n})$ we obtain \[\int_M-\frac{2}{n}\exp\left(-\frac{2v}{n}\right)\omega(dv,dv)\frac{\omega^n}{n!}+\frac{2}{n}\lambda\int_Mg\frac{\omega^n}{n!}=0.\] Using the hypothesis that $\int_Mg\frac{\omega^n}{n!}<0$, we obtain that $\lambda<0.$ So, we can write $-\lambda=\frac{n}{2}\exp(\frac{2\gamma}{n}),$ where $\gamma\in \R$ is an appropriate constant.  Integrating by parts the equation (\ref{eq:11}), we verify that the function $v$ is a solution of\[\Delta_{\omega}^{Ch}v=g\exp\left(\frac{2(v+\gamma)}{n}\right).\] Then,  $u=v+\gamma\in W^{1,2}(M)$ is a solution of (\ref{eq:10}). Again using the Calder\'on-Zygmund inequality, we prove that $u\in C^{\infty}(M,\R)$.  \end{proof}
Notwithstanding the hypothesis of the Theorem above are very strong, we have many examples of compact complex manifolds that satisfy them. Surely, on every compact K\"ahler manifold with $c_1(M)=0$,  we can consider a metric that satisfies the hypothesis of the Theorem \ref{thm:2.11}. In fact, thanks to Yau Theorem, see \cite[Theorem 2]{Yau}, on this type of manifold, we can always choose a Ricci-flat K\"ahler metric. In particular, this metric is balanced and its Chern scalar curvature is identically zero. In the non-K\"ahler case, we can construct balanced metrics with zero Chern scalar curvature on every compact quotient of a holomorphic Lie group. We know that these manifolds are the only compact  complex manifolds that are holomorphically parallelizable, see \cite[Theorem 1]{Wa}. Moreover, every holomorphically parallelizable manifold admits a Chern-flat Hermitian metric, see \cite[Theorem 3]{Bo}. Such metric is balanced and its Chern scalar curvature is equal to zero. More examples are conjectured in \cite[Conjecture 4.1]{Tos}.

Theorem \ref{thm:2.11} allows us to formulate an equivalent condition so that a smooth function changing sign on $M$ can be Chern scalar curvature of a metric conformally equivalent to a metric satisfying the hypothesis of the Thereom \ref{thm:2.11}.\begin{prop}\label{prop:2.12}Let $M^n$ be a connected, compact complex manifold with $n\ge 2$. Suppose that there exists a balanced metric $\omega$ such that $S^{Ch}(\omega)=0$. Then, a function $g\in C^{\infty}(M,\R)$ that changes sign is the Chern scalar curvature of a metric conformally equivalent to $\omega$ if and only if there exists $\varphi\in Aut(M)$ such that\[\int_Mg\circ\varphi\frac{\omega^n}{n!}<0.\]\end{prop} 
\begin{rmk} Unfortunately, we have some examples of functions that cannot be the Chern scalar curvature of a metric conformally equivalent to one satisfying the hypothesis of Theorem \ref{thm:2.11}.  Let $M=\C^n/\Lambda$ be a  complex torus, where $\Lambda$ is a discrete translation subgroup of rank $2n$ acting freely on $\C^n$. We know that \[Aut(M)=GL(\Lambda)\ltimes M,\] where $GL(\Lambda)=\{A\in GL(n, \C) \ | \ A(\Lambda)=\Lambda\}.$ We can consider  the flat metric $\omega$ on $M$. Clearly, $\omega$ is K\"ahler and its Chern scalar curvature is zero. On the other hand, we observe that \[\int_Mg\circ\varphi\frac{\omega^n}{n!}=\int_Mg\frac{(\varphi^{-1})^*\omega^n}{n!}, \ \  \ \forall \varphi\in Aut(M).\] Clearly, if $x\in M$, then $x^*\omega=\omega$, instead, if $A\in GL(\Lambda)$, we have $A^*\omega^n=\lvert\det(A)\rvert\omega^n.$ Moreover, we have that, $\forall \varphi\in Aut(M),$ there exists a unique $x\in M$ and a unique $A\in GL(\Lambda)$ such that $\varphi=x\circ A.$ So \[\int_Mg\circ\varphi\frac{\omega^n}{n!}=\lvert\det(A^{-1})\rvert\int_M g\frac{\omega^n}{n!}.\]Then, a function $g\in C^{\infty}(M,\R)$ that changes sign on $M$ is the Chern scalar curvature of a metric conformally equivalent to $\omega$ if and only if $\int_Mg\frac{\omega^n}{n!}<0$ and then, using Theorem \ref{thm:2.11}, if and only if $g$ is the Chern scalar curvature of a metric conformal to $\omega$.  So, we can construct functions changing sign on $M$ but which are not Chern scalar curvature of any metric conformally equivalent to $\omega$. For example, consider \[\Lambda=\bigoplus_{i=1}^{2n}\Z e_i,\] where $\{e_1,\ldots, e_{2n}\}$ is the standard basis of $\R^{2n}$. If we consider the function $g(x_1,\ldots, x_{2n})=\cos(2\pi x_1)$, we note easily that $\int_Mg\frac{\omega^n}{n!}=0.$ So, $g$ cannot be the Chern scalar curvature of any metric conformal equivalent to $\omega.$ \end{rmk}
 \subsection{Case $\Gamma(\{\omega\})>0$} 
 This last case is the most difficult. We already said that, in this case, we do not know  if there exists a metric with  constant and positive Chern scalar curvature within the  conformal class.  On the other hand, the condition (\ref{eq:2}) guarantees that $g$ must be positive somewhere on $M$.  
The  result we prove below follows from an application of the Implicit function Theorem.\begin{prop}\label{prop:2.14} Let $M^n$ be a connected, compact complex manifold with $n\ge2$ endowed with a Hermitian metric $\omega$ such that $\Gamma(\{\omega\})>0.$ Let $\eta\in \{\omega\}$ be the only Gauduchon metric with volume 1. Then, there exists $\varepsilon>0$, depending on $M$ and $\eta$, such that, if $\lVert S^{Ch}(\eta)\rVert_{C^{0,\alpha}(M)},\lVert g \rVert_{C^{0,\alpha}(M)}<\varepsilon$, where $g\in C^{\infty}(M,\R)$ is positive somewhere on $M$ and $\alpha\in (0,1)$, then $g$ is the Chern scalar curvature of a metric conformal to $\omega.$\end{prop}\begin{proof} Fix $\alpha\in (0,1)$ and consider the Banach manifolds: \[X=\left\{(u,g,S)\in C^{2,\alpha}(M)\times C^{0,\alpha}(M)^2 \ \middle| \ \int_M\left(S-g\exp\left(\frac{2u}{n}\right)\right)\frac{\eta^n}{n!}=0\right\}\]\[Y=\left\{\psi\in C^{0,\alpha}(M) \ \middle| \int_M\psi\frac{\eta^n}{n!}=0\right\}.\] Consider the map $F\colon X\to Y$ such that\[ F(u,g,S)=\Delta_{\eta}^{Ch}u+S-g\exp\left(\frac{2u}{n}\right), \ \ \forall (u,g,S)\in X.\] Clearly, $F$ is $C^1$. Observe that $F(0,0,0)=0$ and that\[\frac{\partial F}{\partial u}_{|(0,0,0)}(\phi)=\Delta_{\eta}\phi.\] is invertible on $T_0Y\simeq Y.$ So, using the implicit function Theorem, there exists a neighbourhood $U\subset C^{0,\alpha}(M)^2$ of $(0,0)$ and a $C^1$ function $\Phi\colon U\to X$ such that \[F(\Phi(g,S),g,S)=0,\ \ \forall (g,S)\in U.\] Choose $\varepsilon>0$, depending only on $M$ and $\eta$, such that \[\left\{g\in C^{0,\alpha}(M) \ \middle| \  \max_Mg>0,  \ \lVert g \rVert_{C^{0,\alpha}(M)}<\varepsilon\right\}\times B_{C^{0,\alpha}(M)}(0,\varepsilon)\subseteq U.\] So, if $\eta $ is such that $\lVert S^{Ch}(\eta)\rVert_{C^{0,\alpha}(M)}<\varepsilon$ and $g\in C^{\infty}(M,\R)$ somewhere positive on $M$ and such that $\lVert g \rVert_{C^{0,\alpha}(M)}<\varepsilon$, we can choose $u=\Phi(g,S)$ and verify that $S^{Ch}(\exp\left(\frac{2u}{n}\right))=g$. By Schauder estimates, we obtain that $u\in C^{\infty}(M,\R).$ \end{proof}
The Proposition above is a generalization of \cite[Theorem 5.9]{AnCaSp}. We observe that the condition on $S^{Ch}(\eta)$ in Proposition \ref{prop:2.14} is expressed in terms of $\varepsilon$, which depends itself on $\eta$. So, a priori, the result may be empty. However, in \cite[Remark 5.10]{AnCaSp}, the authors present some  examples where the \cite[Theorem 5.9]{AnCaSp} can be applied. In these examples, Proposition \ref{prop:2.14} can also be applied.

\end{document}